\documentclass[11pt]{amsart}
\usepackage[margin=1.5in]{geometry}

\usepackage{pinlabel}
         
\usepackage[all,cmtip]{xy}
\usepackage{amsmath,amssymb,amsthm, amscd, xypic}
\usepackage{mathtools}
\usepackage[english]{babel}
\usepackage{tikz-cd, enumerate}

\usepackage[normalem]{ulem}%for coloured strikeouts

\usepackage{todonotes}

\newtheorem{introthm}{Theorem}

\usepackage{xcolor}

\colorlet{GREEN}{green}
\colorlet{BLUE}{blue}

\definecolor{darkgreen}{rgb}{0,0.50,0} 
\definecolor{darkred}{rgb}{0.55,0,0}
\definecolor{darkblue}{rgb}{0,0,0.6}

\usepackage[pdfborder=0,pagebackref,colorlinks,citecolor=darkgreen,linkcolor=darkred,urlcolor=darkblue]{hyperref}
\addto\extrasenglish{
    
}
\def\makeautorefname#1#2{\expandafter\def\csname#1autorefname\endcsname{#2}}
\makeautorefname{equation}{Equation}%
\makeautorefname{footnote}{footnote}%
\makeautorefname{item}{item}%
\makeautorefname{figure}{Figure}%
\makeautorefname{table}{Table}%
\makeautorefname{part}{Part}%
\makeautorefname{appendix}{Appendix}%
\makeautorefname{chapter}{Chapter}%
\makeautorefname{section}{Section}%
\makeautorefname{subsection}{Section}%
\makeautorefname{subsubsection}{Section}%
\makeautorefname{paragraph}{Paragraph}%
\makeautorefname{subparagraph}{Paragraph}%
\makeautorefname{theorem}{Theorem}%
\makeautorefname{thm}{Theorem}%
\makeautorefname{introthm}{Theorem}%
\makeautorefname{addm}{Addendum}%
\makeautorefname{mainthm}{Main theorem}%
\makeautorefname{corollary}{Corollary}%
\makeautorefname{cor}{Corollary}%
\makeautorefname{lemma}{Lemma}%
\makeautorefname{lem}{Lemma}%
\makeautorefname{sublemma}{Sublemma}%
\makeautorefname{sublem}{Sublemma}%
\makeautorefname{subl}{Sublemma}%
\makeautorefname{prop}{Proposition}%
\makeautorefname{property}{Property}
\makeautorefname{pro}{Property}
\makeautorefname{sch}{Scholium}%
\makeautorefname{step}{Step}%
\makeautorefname{conject}{Conjecture}%
\makeautorefname{conj}{Conjecture}%
\makeautorefname{questn}{Question}
\makeautorefname{quest}{Question}
\makeautorefname{qn}{Question}
\makeautorefname{definition}{Definition}%
\makeautorefname{defn}{Definition}%
\makeautorefname{defi}{Definition}%
\makeautorefname{def}{Definition}%
\makeautorefname{dfn}{Definition}%
\makeautorefname{notation}{Notation}
\makeautorefname{notn}{Notation}
\makeautorefname{rem}{Remark}%
\makeautorefname{rems}{Remarks}%
\makeautorefname{rmk}{Remark}%
\makeautorefname{rk}{Remark}%
\makeautorefname{remarks}{Remarks}%
\makeautorefname{rems}{Remarks}%
\makeautorefname{rmks}{Remarks}%
\makeautorefname{rks}{Remarks}%
\makeautorefname{example}{Example}%
\makeautorefname{examp}{Example}%
\makeautorefname{exmp}{Example}%
\makeautorefname{exam}{Example}%
\makeautorefname{exa}{Example}%
\makeautorefname{axiom}{Axiom}%
\makeautorefname{axi}{Axiom}%
\makeautorefname{ax}{Axiom}%
\makeautorefname{case}{Case}%
\makeautorefname{claim}{Claim}%
\makeautorefname{clm}{Claim}%
\makeautorefname{assumpt}{Assumption}%
\makeautorefname{asses}{Assumptions}%
\makeautorefname{conclusion}{Conclusion}%
\makeautorefname{concl}{Conclusion}%
\makeautorefname{conc}{Conclusion}%
\makeautorefname{cond}{Condition}%
\makeautorefname{const}{Construction}%
\makeautorefname{con}{Construction}%
\makeautorefname{criterion}{Criterion}%
\makeautorefname{criter}{Criterion}%
\makeautorefname{crit}{Criterion}%
\makeautorefname{exercise}{Exercise}%
\makeautorefname{exer}{Exercise}%
\makeautorefname{exe}{Exercise}%
\makeautorefname{problem}{Problem}%
\makeautorefname{problm}{Problem}%
\makeautorefname{prob}{Problem}%
\makeautorefname{prob}{Problem}%
\makeautorefname{soln}{Solution}%
\makeautorefname{sol}{Solution}%
\makeautorefname{sum}{Summary}%
\makeautorefname{oper}{Operation}%
\makeautorefname{obs}{Observation}%
\makeautorefname{ob}{Observation}%
\makeautorefname{conv}{Convention}%
\makeautorefname{cvn}{Convention}%
\makeautorefname{warn}{Warning}%
\makeautorefname{note}{Note}%
\makeautorefname{fact}{Fact}%
\makeautorefname{ouch0}{Counterexample}%
\newtheorem{thm}{Theorem}[section]
\newtheorem{theorem}{Theorem}[section]

\newtheorem{prop}{Proposition}[section]

\newtheorem{lemma}{Lemma}[section]

\theoremstyle{definition}
\newtheorem{definition}{Definition}[section]
\newtheorem{defn}{Definition}[section]

\newtheorem{exmp}{Example}[section]

\newtheorem{rem}{Remark}[section]

\newtheorem{sch}{Scholium}[section]

\makeatletter
\let\c@cor=\c@thm
\let\c@prop=\c@thm
\let\c@proposition=\c@thm
\let\c@theorem=\c@thm
\let\c@lem=\c@thm
\let\c@definition=\c@thm
\let\c@conj=\c@thm
\let\c@defn=\c@thm
\let\c@df=\c@thm
\let\c@exmp=\c@thm
\let\c@rem=\c@thm
\let\c@lemma=\c@thm
\let\c@sch=\c@thm
\let\c@con=\c@thm
\let\c@equation\c@thm
\makeatother

\newcommand{\im}{\mathrm{Im}}

\newcommand{\SK}{\mathrm{SK}}
\newcommand{\SKK}{\mathrm{SKK}}
\newcommand{\id}{\mathrm{id}}

\def\C{\mathcal{C}}
\def\d{\partial}

\newcommand{\R}{\mathbb R}
\newcommand{\Z}{\mathbb Z}

\newcommand{\Mnfldbd}{\textup{Mfd}^\partial}

\newcommand{\Fun}{\textup{Fun}}

\newcommand{\ChB}{\textup{Ch}^{b}}
\newcommand{\ModFG}{\textup{Mod}^{\textup{fg}}}
\newcommand{\ModPROJ}{\textup{Mod}^{\textup{proj}}}
\newcommand{\ChPerf}{\textup{Ch}^{\textup{hb}}}

\newcommand{\sC}{\mathcal{C}}

\newcommand{\mona}[1]{{\color{black}{#1}}}
\newcommand{\carmen}[1]{{\color{black}{#1}}}
\newcommand{\laura}[1]{{\color{black}{#1}}}
\newcommand{\julia}[1]{{\color{black}{#1}}}

\newcommand{\lout}{\bgroup\markoverwith
	{\textcolor{blue}{\rule[.5ex]{2pt}{0.4pt}}}\ULon}
\newcommand{\jout}{\bgroup\markoverwith
	{\textcolor{blue}{\rule[.5ex]{2pt}{0.4pt}}}\ULon}
	
\newcommand{\cout}{\bgroup\markoverwith
	{\textcolor{blue}{\rule[.6ex]{3pt}{0.6pt}}}\ULon}
	
	\newcommand{\mout}{\bgroup\markoverwith
	{\textcolor{blue}{\rule[.6ex]{3pt}{0.6pt}}}\ULon}

\newcommand\rout{\bgroup\markoverwith
	{\textcolor{blue}{\rule[.5ex]{2pt}{0.4pt}}}\ULon}

\newcommand{\sbt}{\,\begin{picture}(-1,1)(0.5,-1)\circle*{1.8}\end{picture}\hspace{.05cm}}

\newlength{\storeparskip}
\author[R. S. Hoekzema, M. Merling, L. Murray, C. Rovi and J. Semikina]{Renee  S. Hoekzema, Mona Merling, Laura Murray, \\ Carmen Rovi and Julia Semikina}

\date{}

\title{Cut and paste invariants of manifolds via algebraic $K$-theory}

\begin{document}

\begin{abstract} Recent work of Inna Zakharevich and Jonathan Campbell has focused on building machinery for studying scissors congruence problems via algebraic $K$-theory, and applying these tools to studying the Grothendieck ring of varieties. In this paper we give a new application of their framework: we construct a $K$-theory space that recovers the classical $\SK$ (``schneiden und kleben,"  German for ``cut and paste") groups for manifolds on $\pi_0$, and we construct a derived version of the Euler characteristic.   

\end{abstract}

\maketitle

\vspace{-2ex}

\begingroup%
\setlength{\parskip}{\storeparskip}% Restore \parskip within this scope
\tableofcontents
\endgroup%

\setcounter{section}{0}

\section{Introduction}
The classical scissors congruence problem asks whether given two polyhedra with the same volume $P$ and $Q$ in $\R^3$, one can cut $P$ into a finite number of smaller polyhedra and reassemble these to form $Q$. Precisely, $P$ and $Q$ are scissors congruent if $P=\bigcup_{i=1}^m P_i$ and $Q=\bigcup_{i=1}^m Q_i$, where $P_i\cong Q_i$ for all $i$, and the subpolyhedra in each set only intersect each other at edges or faces. There is an analogous definition of an $\SK$ (German ``schneiden und kleben," cut and paste)  relation for manifolds: Given a closed smooth oriented manifold $M$, one can cut it along a separating codimension $1$ submanifold $\Sigma$ with trivial normal bundle and paste back the two pieces along an orientation preserving diffeomorphism $\Sigma\rightarrow \Sigma$ to obtain a new manifold, which we say is ``cut and paste equivalent" or ``scissors congruent" to it. We give a pictorial example of this relation:
\vspace{15pt}
\begin{figure}[ht!]
\labellist
\small\hair 2pt
%%%%%%%%%%%%    FIRST LABEL
\pinlabel \small {\textit{1. Start with} $T^2$} at 72 330
%\pinlabel \small {$T^2 \sqcup T^2$} at 77 248
%%%%%%%%%%%%    SECOND LABEL
\pinlabel \small {\textit{2. Cut along four copies of $S^1$} } at 327 330
%%%%%%%%%%%%    THIRD LABEL
\pinlabel \small { \textit{3. Paste back along boundaries}} at 637 330
\endlabellist
\centering
\includegraphics[scale=0.5]{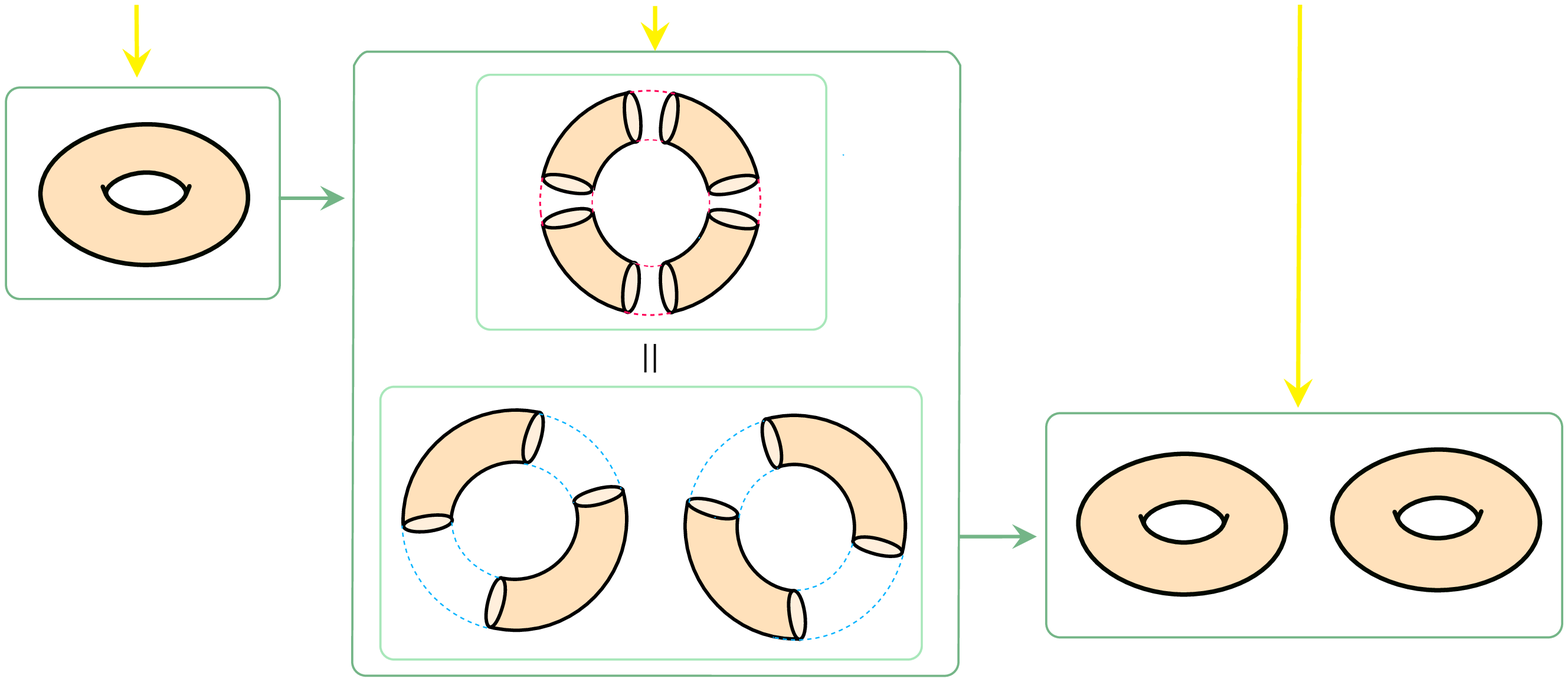}
\caption{Example of a cut and paste operation}
\label{fig:SK-equiv}
\end{figure}

Zakharevich has formalized the notion of scissors congruence via the notion of an \emph{assembler}--this is a Grothendieck site with a few extra properties, whose topology encodes the cut and paste operation. She constructs an associated $K$-theory spectrum, which on $\pi_0$ recovers classical scissors congruence groups \cite{Inna-assemblers}. Specific examples of assemblers  recover scissors congruence groups for polytopes and the Grothendieck ring of varieties, as $\pi_0$ of their corresponding $K$-theory spectra. The higher $K$-groups encode further geometric information. Independently, Campbell has  introduced the formalism of \emph{subtractive categories}, a modification of the definition of Waldhausen categories, to define a $K$-theory spectrum of varieties that recovers the Grothendieck ring of varieties on $\pi_0$ \cite{Campbell}. Though the approaches to encoding  scissors congruence abstractly are different, the resulting spectra of Zakharevich and Campbell are shown to be equivalent in \cite{inna-cgw}.
 
 The focus of Zakharevich and Campbell has been to construct and study a $K$-theory spectrum of varieties, and this spectrum level lift of the Grothendieck ring of varieties has led to a fruitful research program to better understand varieties. For example, an analysis of $K_1$ for the $K$-theory spectrum of varieties allowed Zakharevich to elucidate structure on the annihilator of the Lefschetz motive \cite{inna-lefschetz}, and Campbell, Wolfson and Zakharevich use a lift of the zeta function for varieties  to show that $\pi_1$ of the $K$-theory spectrum for varieties contains nontrivial geometric information  \cite{CWZ}. Studying cut and paste relations for manifolds via  $K$-theoretic machinery remains as of yet unexplored. We start this exploration in this paper.
 
 Unfortunately, the framework from \cite{Inna-assemblers, Campbell} does not directly apply to the case of manifolds. The problem   is that if one tries to find a common refinement of two different SK-decompositions of a manifold, one might have to cut boundaries and one gets manifolds with corners. This makes some of the axioms in both the assembler approach and the subtractive category approach break down. However, work in progress of Campbell and Zakharevich  on ``$K$-theory with squares," $K^\square$, a further synthetization of scissors congruence relations as $K$-theory that generalizes Waldhausen $K$-theory, does give the right framework to construct the desired scissors congruence $K$-theory for manifolds. Encompassing the manifold example was also one of the motivations behind Campbell and Zakharevich's development of ``$K$-theory with squares". 

The study of $\SK$-invariants and $\SK$-groups in \cite{KKNO} focuses on closed manifolds. However, in order for the $K^{\square}$-theoretic scissors congruence machinery to apply, we need  to work in the category of manifolds with boundary, since the pieces in an $\SK$-decomposition have boundary. This is not well-explored classically, as most of the existing work on $\SK$-groups is for closed manifolds. We generalize the notion of $\SK$-equivalence to the case of manifolds with boundary and denote the corresponding group by $\SK^{\partial}_{n}.$ Our definition of  $\SK^{\partial}_{n}$ is different from the one mentioned in \cite{KKNO} in that we insist that every boundary along which  we cut gets pasted, and this is crucial for the further application of the $K$-theoretic technology. 

We formulate a suitable notion of a category with squares $\Mnfldbd_n$, that fits into the $K$-theory with squares framework, and whose  distinguished squares  exactly encode the ``cut-and-paste" relations for $n$-dimensional manifolds with boundary. We show that the \mona{$K$-theory space}  obtained from the construction of Campbell and Zakharevich, applied to $\Mnfldbd_n$, which we denote by $K^{\square}(\Mnfldbd_n)$, recovers the $\SK^{\partial}_n$ as its zeroth homotopy group:

 \begin{introthm}\label{k0intro}
 There is an isomorphism $K_0^{\square}(\Mnfldbd_n) \cong \SK^{\partial}_n,$ where $K_0^{\square}(\Mnfldbd_n)$ is $\pi_0$ of a scissors congruence $K$-theory \mona{space} $K^{\square}(\Mnfldbd_n)$.
 \end{introthm}

% For closed manifolds,  there is a more refined notion of  $\SKK$-invariance which  differs from  $\SK$-invariance by a controlled correction term that is allowed to depend only on the gluing diffeomorphisms but not the cut submanifold pieces. The $\SKK$-groups can be interpreted as Reinhardt vector field bordism groups \cite{KKNO}, which equivalently can be seen to be $\pi_0$ of the Madsen-Tillman spectrum $MTSO(n)$  \cite{Ebert}, or $\pi_1$ of the cobordism category. We give a definition of $\SKK$-groups for manifolds with boundary, and the conjecture, which we will investigate in future work, is that they arise as $\pi_1$ of $K^{\square}(\Mnfldbd_n)$. This expectation is inspired by discussions with Inna Zakharevich and Jonathan Campbell and is reminiscent of results on $K_1$ that Zakharevich has obtained in other contexts.

Scissors congruence invariants for manifolds ($\SK$-invariants) are abelian group valued homomorphisms from the monoid of manifolds under disjoint union, which factor through the $\SK$-group. It is well known classically that for closed manifolds  the Euler characteristic, the signature, and linear combinations thereof, are the only $\SK$-invariants, and these are also $\SK$-invariants of manifolds with boundary. In this paper, we show that the Euler characteristic as a map to $\Z$, viewed as the zeroth $K$-theory group of $\Z$, is the $\pi_0$ level of a map of \mona{$K$-theory spaces} from the scissors congruence \mona{space} for manifolds with boundary that we define. 
In future work, we plan to also investigate the signature map to the zeroth $L$-theory group of $\Z$.
  \begin{introthm}\label{Eulerthminto}
 There is a map of $K$-theory \mona{spaces} $$K^\square(\Mnfldbd)\to K(\Z),$$ which on $\pi_0$ agrees with the Euler characteristic for smooth compact manifolds with boundary.
 \end{introthm}
 
All the scissors congruence space constructions are in fact infinite loop spaces, and it is not hard to see that all of our maps of $K$-theory spaces lift to the spectrum level. The spectrum level elborations will appear in future work of Campbell, Zakharevich and collaborators.

The paper is organized as follows. In \autoref{SKdefsec} we introduce the definition of $\SK$-groups %and $\SKK$-groups 
for smooth compact manifolds with boundary and we prove that they are related to the classical $\SK$ %and $\SKK$ 
groups for smooth closed manifolds via an exact sequence. In \autoref{squareK} we review the set-up of categories with squares and their $K$-theory as defined by Campbell and Zakharevich. In \autoref{mnfldsec} we construct the category of squares for smooth compact manifolds with boundary and prove  \autoref{k0intro}, and in \autoref{Eulersec} we prove \autoref{Eulerthminto}.

\subsection*{Conventions} All manifolds in this paper are smooth, compact and oriented. We will distinguish between closed manifolds and manifolds with boundary. We will use the notation $\bar{M}$ for the manifold $M$ with reversed orientation.

\subsection*{Acknowledgements} 
We are greatly indebted  to Jonathan Campbell and Inna Zakharevich for their generosity in sharing their work in progress on squares $K$-theory, which our project relies on, and for their extensive patience in explaining it to us and answering our questions.  \mona{We thank the anonymous referee and Jim Stasheff for close readings and suggestions that greatly improved the paper. We particularly thank George Raptis for catching a mistake in a previous version of the paper regarding our definition of an $SKK$ group for manifolds with boundary---we owe the observation in \autoref{scholium} to him.}
It is a great pleasure to also acknowledge the contributions to this project arising from discussions with Andrew Blumberg, Anna Marie Bohmann, Jonathan Block, Jim Davis, Greg Friedman, S\o ren Galatius, Teena Gerhardt, Herman Gluck, Mike Hill, Fabian Hebestreit, Matthias Kreck, Malte Lackmann, Wolfgang L{\"u}ck, Mike Mandell, Cary Malkiewich, Kate Ponto, Jan Steinebrunner, Peter Teichner, Andrew Tonks, and Chuck Weibel. 

Finally, we thank the organizers of the Women in Topology III program and the Hausdorff Research Institute for Mathematics for their hospitality during the workshop. The WIT III workshop was supported through grants NSF-DSM 1901795, NSF-HRD 1500481 - AWM ADVANCE grant
and the Foundation Compositio Mathematica, and we are very grateful for their support. 
The first named author was supported by the Max Planck Society.
The second named author was  supported by grants NSF-DMS 1709461, NSF CAREER DMS 1943925 and NSF FRG DMS-2052988. The third named author was  supported by grant NSF-DMS 1547292. 
The fourth named author is supported by the Deutsche Forschungsgemeinschaft (DFG, German Research Foundation) under Germany’s Excellence Strategy EXC2181/1-390900948 (the Heidelberg STRUCTURES Excellence Cluster), and also wishes to acknowledge support by the Deutsche Forschungsgemeinschaft (DFG, German Research Foundation) 281869850 (RTG 2229).
The fifth named author was supported by the Max Planck Society and Wolfgang L\"uck’s ERC Advanced Grant ``KL2MG-interactions'' (no. 662400).

\section{Scissors congruence groups for manifolds with boundary}\label{SKdefsec}

\subsection{SK-groups for closed manifolds} We start by reviewing the definitions of the classical scissors congruence groups of smooth closed  oriented manifolds, namely the $\SK_n$-groups introduced in \cite{KKNO}. The ``scissors congruence" or ``cut and paste" relation on smooth closed oriented  manifolds is given as follows: cut an $n$-dimensional manifold $M$ along a codimension 1 smooth submanifold  $\Sigma$ with trivial normal bundle that separates $M$ in the sense that the complement of $\Sigma$ in $M$ is a disjoint union of two components $M_1$ and $M_2$, each with boundary diffeomorphic to $\Sigma$. Then paste back the two pieces together along an orientation preserving diffeomorphism $\phi\colon \Sigma\to\Sigma.$ We say $M$ and $M_1\cup_\phi M_2$ are ``cut and paste equivalent" or ``scissors congruent." 

Note that for a codimension 1 submanifold $\Sigma$ with trivial normal bundle that does not separate $M$ (for example the inclusion of $S^1\times\{0\}$ into $S^1\times S^1$) we can take the union with a second copy of  $\Sigma$ embedded close to it, and the disjoint union $\Sigma \sqcup \Sigma$ then separates $M$. 

\begin{definition}
Two smooth closed manifolds $M$ and $N$ are \emph{$SK$-equivalent} (or \emph{scissors congruent} or \emph{cut and paste equivalent}) if $N$ can be obtained from $M$ by a finite sequence of cut and paste operations.
\end{definition}

\begin{exmp}
In  \autoref{SK-relation-closed} we can see that $T^2 ~ \sharp ~ T^2 \sqcup S^2$   is SK-equivalent to $T^2 \sqcup T^2$. 
\end{exmp}
\begin{figure}[ht!]
\labellist
\small\hair 2pt
%%%%%%%%%%%%    FIRST morph
\pinlabel \small {$\phi  \downarrow$} at -5 92
%%%%%%%%%%%%    SECOND morph
\pinlabel \small {$\psi  \downarrow$} at 195 92
%%%%%%%%%%%%   simeq
\pinlabel \small {$\simeq$} at 153 92
%%%%%%%%%%%%    SK
\pinlabel \small {$\in \SK_2$} at 356 92
\endlabellist
\centering
\includegraphics[scale=0.7]{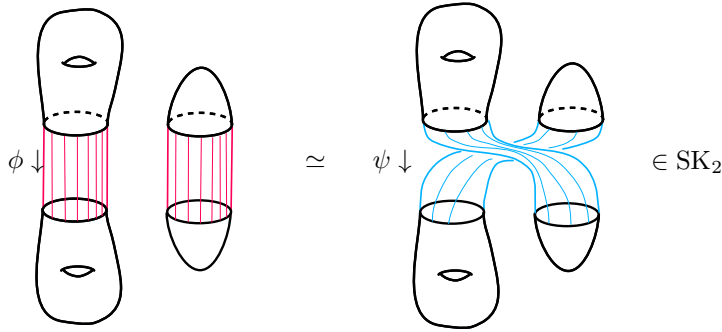}
\caption{Example of an $\SK$-relation}
\label{SK-relation-closed}
\end{figure}

%\vspace{25pt}
%\begin{figure}[ht!]
%\labellist
%\small\hair 2pt
%%%%%%%%%%%%    FIRST LABEL
%\pinlabel \small {\textit{1. Start with}} at 77 275
%\pinlabel \small {$T^2 \sqcup T^2$} at 77 248
%%%%%%%%%%%%    SECOND LABEL
%\pinlabel \small {\textit{2. Cut along} $S^0\times S^1$} at 327 248
%%%%%%%%%%%%    THIRD LABEL
%\pinlabel \small { \textit{3. Paste back along boundaries:}} at 607 248
%\endlabellist
%\centering
%\includegraphics[scale=0.6]{SK-inverse2}
%\caption{Example of cut and paste operation}
%\label{fig:SK-equiv}
%\end{figure}

%\begin{figure}[ht!]
%\labellist
%\small\hair 2pt
%%%%%%%%%%%%%    FIRST LABEL
%\pinlabel \small {\textit{1. Start with}} at 77 275
%\pinlabel \small {$(S^1\times S^1) \sqcup (S^1 \times S^1)$} at 77 248
%%%%%%%%%%%%%    SECOND LABEL
%\pinlabel \small {\textit{2. Cut along} $S^0\times S^1$} at 327 248
%%%%%%%%%%%%%    THIRD LABEL
%\pinlabel \small { \textit{3. Paste back along boundaries:}} at 607 248
%\endlabellist
%\centering
%\includegraphics[scale=0.5]{S1SK}
%\caption{Example of cut and paste operation}
%\label{fig:SK-equiv}
%\end{figure}

%
%\mmar{put pic of SK relation we discussed here maybe}
%\cmar{We could also use the figure from the proposal}

 Let $\mathcal{M}_n$ be the monoid of diffeomorphism classes of smooth closed oriented $n$-dimensional manifolds $[M]$ under disjoint union. The  $\SK_n$-group from \cite{KKNO}  is defined to satisfy the universal property that any abelian valued monoid map from $\mathcal{M}_n$ which respects $\SK$-equivalence (also called an $\SK$-invariant) factors through it. 

\begin{definition} \label{def of SK_n}
 The \emph{scissors congruence group $\SK_n$ for smooth closed  oriented $n$-dimensional manifolds} is the quotient of the Grothendieck group $\textup{Gr}(\mathcal{M}_n)$ by the $\SK$-equivalence relation. 
  
  Explicitly, $\SK_n$ is the free abelian group on diffeomorphism classes  $[M]$ modulo the following relations:
   \begin{enumerate}
    \item $[M \sqcup N] = [M] + [N]$;
     \item Given compact oriented manifolds $M_1, M_2$ and orientation preserving diffeomorphisms $\phi, \psi \colon \partial M_1 \to \partial M_2$,
\[ [M_1 \cup_{\phi} \bar{M_2}] = [M_1 \cup_{\psi} \bar{M_2}],\]
where $\bar{M_2}$ is $M_2$ with reversed orientation.
   \end{enumerate}
\end{definition}
%\rmar{Closed means compact with no boundary, so doesn't apply to the parts $M$ and $M$', but it does to the generators in this setting. }

%\lmar{Should we put in a picture here as well? (Maybe the first half of the SKK picture below?)}

\carmen{\begin{rem}\label{SKinvariants} It is shown in \cite[Corollary 1.4]{KKNO} that any SK-invariant for smooth oriented manifolds is a linear combination of the Euler characteristic and the signature.  \end{rem}}

\subsection{SK-groups for  manifolds with boundary}
We note that in order to define a scissors congruence $K$-theory, we need to work in a category of manifolds with boundary since the pieces in the cut and paste relation are manifolds with boundary. Therefore, we introduce a definition of $\SK$-groups for manifolds with boundary; these are the groups  which we will recover as $\pi_0$ of a scissors congruence $K$-theory space.

We define the ``cut and paste relation" on smooth compact manifolds with boundary analogously to that on closed manifolds: cut an $n$-dimensional manifold $M$ along a codimension 1 smooth submanifold  $\Sigma$ with trivial normal bundle, which separates $M$, and for which $\Sigma\cap \partial M=\emptyset$. Then paste back the two pieces together along an orientation preserving diffeomorphism $\phi\colon \Sigma\to\Sigma.$ We emphasize that we do not allow boundaries to be cut, and we require that  all boundaries which come from cutting to be pasted back together, leaving the existing boundaries of a manifold untouched by the cut and paste operation. 

\begin{definition} Two smooth compact manifolds with boundary will be called \emph{$\SK$-equivalent} if one can be obtained from the other via a finite sequence of cut and paste operations in the sense described above.\end{definition}

\begin{rem}
Our definition of the cut and paste relation for manifolds with boundary is different than the one in \cite[Chapter 5]{KKNO}, where $M_1\cup_\phi M_2\sim M_1\sqcup M_2$. Namely, they allow pieces that are cut to not be pasted back together. 
 In order to apply the $K$-theoretic machinery to obtain the $\SK_n^{\partial}$-group as $\pi_0$ of a $K$-theory space, it is important to use our definition of $\SK^{\partial}_n$.
\end{rem}

\begin{definition} \label{def of SK_n with boundary}
Let $\mathcal{M}_n^\d$ be the monoid of diffeomorphism classes of  smooth compact oriented $n$-dimensional manifolds with boundary  under disjoint union. The \emph{scissors congruence group $\SK_n^\d$ for smooth compact oriented manifolds} is the quotient of the Grothendieck group $\mathrm{Gr}(\mathcal{M}_n^\d)$ by the $\SK$-equivalence relation. 

Explicitly, $\SK^{\partial}_n$ is the free abelian group on diffeomorphism classes of smooth compact oriented $n$-dimensional manifolds (with or without boundary) modulo the following relations:
  \begin{enumerate}
    \item $[M \sqcup N] \sim [M] + [N]$;
     \item Given compact oriented manifolds $M_1, M_2$, closed submanifolds $\Sigma \subseteq \partial M_1$ and $\Sigma' \subseteq \partial M_2$, and orientation preserving diffeomorphisms $\phi, \psi \colon \Sigma \to \Sigma'$,
\[ [M_1 \cup_{\phi} \bar{M_2}] = [M_1 \cup_{\psi} \bar{M_2}].\]
   \end{enumerate}
\end{definition}

\begin{exmp} In \autoref{SK-relation-boundary} we see an example of an $SK_n^{\partial}$-relation.
\end{exmp}

\begin{figure}[ht!]
\labellist
\small\hair 2pt
%%%%%%%%%%%%    FIRST morph
\pinlabel \small {$\phi  \downarrow$} at 70 154
%%%%%%%%%%%%    SECOND morph
\pinlabel \small {$\psi  \downarrow$} at 490 154
%%%%%%%%%%%%    SIMEQ
\pinlabel \small {$\simeq$} at 370 154
%%%%%%%%%%%%    SK
\pinlabel \small {$\in \SK^{\partial}_n$} at 810 154
\endlabellist
\centering
\includegraphics[scale=0.42]{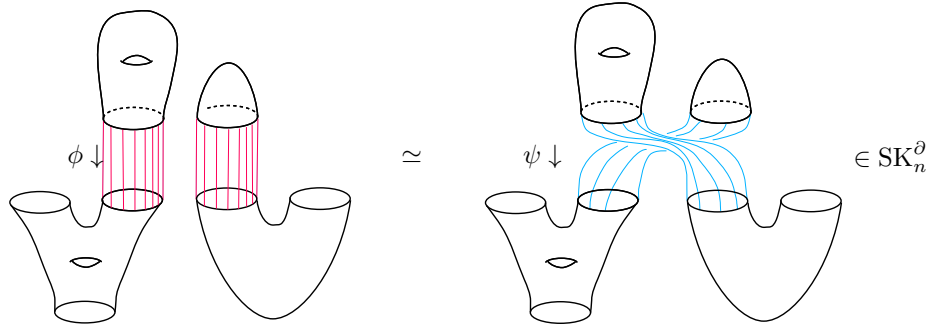}
\caption{Example of an $\SK^{\partial}$-relation}
\label{SK-relation-boundary}
\end{figure}

\carmen{\begin{prop}\label{SKbinvariants} The Euler characteristic and the signature are $\SK^{\partial}$- invariants.\end{prop}

\begin{proof}
That the signature of manifolds with boundary is a $\SK^{\partial}$-invariant follows from Novikov additivity \cite{Novikov} and the fact that in our definition of the $\SK^{\partial}$-relation the boundary components of $\SK^{\partial}$-equivalent manifolds remain unchanged by the cut and paste operation.

The argument that shows that the Euler characteristic is a $\SK^{\partial}$-invariant is an exact analog to the argument for closed manifolds.
\end{proof}}

%\jnote{Visually Figure 3 looks great, but again it is a bit unfortunate relation, since on both sides we have the same manifold...}

\subsection{Exact sequence for $SK$ and $SK^\partial$} We now relate our definition of $\SK_n^\d$ with the classical $\SK_n$ via an exact sequence. Denote by $C_n$ the Grothendieck group of the monoid of diffeomorphism classes of smooth closed oriented $n$-dimensional nullcobordant manifolds under disjoint union.
%that bound a manifold of dimension $n+1$ 
%modulo relations $M \sqcup N \sim M+N$. \rout{It is a free abelian group generated by the diffeomorphism classes of closed connected oriented $n$-manifolds.}\rmar{Not all mfds are bounding.}
\julia{Note that $C_n$ is a free abelian group for the following reason. The Grothendieck group $\mathrm{Gr}(\mathcal{M}_n)$ is a free abelian group on diffeomorphism classes of  connected  $n$-manifolds. Due to the cancellation property, the inclusions of the monoid of diffeomorphism classes of $n$-dimensional nullcobordant manifolds into its Grothendieck group and into $\mathrm{Gr}(\mathcal{M}_n)$ are injective.
Hence the Grothendieck group $C_n$ is a subgroup of $\mathrm{Gr}(\mathcal{M}_n)$ and therefore is free abelian.}

%\renee{, despite the fact that it is in general not generated by connected manifolds.}

 \begin{theorem} \label{Exact SK-sequence}
 For every $n \geq 1$ the following sequence is exact 
 
   \begin{center}
    $\begin{CD}
     0 @>>> \SK_{n} @>\alpha>{[M] \mapsto [M]}> \SK^{\partial}_{n} @>\beta>{[N] \mapsto [\partial N]}> C_{n-1} @>>>0
     \end{CD}.$
    \end{center} 
 \end{theorem}
   \begin{proof}
%\bullet{} 
 Note that the map $\alpha \colon \SK_{n} \to \SK^{\partial}_{n}$ taking a class of manifolds in $\SK_n$ to a class containing the same manifolds in $\SK^{\partial}_{n}$ is well-defined, since every relation from the definition of $\SK_n$ is also a relation in the definition of $\SK^{\partial}_n.$  The map $\beta$ that takes a class of manifolds to the diffeomorphism class of the boundary is well-defined, since the equivalence relation from the definition of $\SK^{\partial}_n$ preserves the boundary.

 We show exactness at the middle term. It is clear from the definition that $\im ~  \alpha \subseteq \ker \beta.$ Let us show the reverse inclusion. Let $x \in \ker \beta.$ Every element of $\SK^{\partial}_n$ can be written in the form $x=[M]-[N],$ where $M, N$ are compact smooth oriented $n$-manifolds with boundary (not necessarily connected). 

Let $\bar{M}$ be the copy of $M$ with the opposite orientation and let $DM$ be the double of $M,$ i.e. $DM=M \cup_{\id} \bar{M}.$ Note that $DM$ is a closed manifold. Since $C_{n-1}$ is a free abelian group and $\beta(x)=[\partial M]-[\partial N] =0$ we conclude that the $\partial M$ and $\partial N$ are diffeomorphic. Hence we may glue $\bar{M}$ to $N$ along the boundary.
%\rout{ using the identity morphisms on the corresponding connected components}\rmar{There is no identity, only a diffeo, right?}. 
We will call this gluing diffeomorphism $\phi$ (it does not have to be unique, we just pick one) and denote by $L$ the closed manifold, which is the result of this gluing. Therefore,
  \[
  DM=M \cup_{\id} \bar{M},  
  \] and
  \[
  L=N \cup_{\phi} \bar{M}.
  \]
Hence in $\SK^{\partial}_n$, 
 \begin{eqnarray*}
 [N] + [DM] &=&[N \cup_{\id} (\partial N \times [0,1])] + [M \cup_{\id} \bar{M}]\\
   &=&[N \cup_{\phi} \bar{M} ]+ [(\partial N \times [0,1]) \cup_{\phi} M]\\
&=& [L] + [M].
 \end{eqnarray*}
 Consequently,
  \[
    x=[M]-[N]=[DM]-[L] \in \im \alpha.
     \] See \autoref{image-alpha} for an illustration of such an element.

 Finally, let us show injectivity of the map $\alpha$. Let $R_n$ be the subgroup of $\mathrm{Gr}(\carmen{\mathcal{M}_n})$ generated by the $\SK$-relation $ [ M_1 \cup_{\phi} \bar{M_2} ] - [M_1 \cup_{\psi} \bar{M_2}] $, 
 so that $\SK_n=\textup{Gr}(\mona{\mathcal{M}_n})/R_n$.  Note that the set of elements that generate this relation is closed under summation,
  \begin{eqnarray*}
     &\big( [ M_1 \cup_{\phi} \bar{M_2} ] - [M_1 \cup_{\psi} \bar{M_2}] \big) + \big( [M'_1 \cup_{\phi'} \bar{M'_2}] - [M'_1 \cup_{\psi'} \bar{M'_2}] \big)\\
     & =[ (M_1 \sqcup M'_1)  \cup_{\phi \sqcup \phi'} (\bar{M_2} \sqcup \bar{M'_2})] - [ (M_1 \sqcup M'_1)  \cup_{\psi \sqcup \psi'} (\bar{M_2} \sqcup \bar{M'_2})]. 
    \end{eqnarray*}
    Thus $R_n$ is precisely the set of elements of this form, and similarly for the subgroup $R^\d_n$ of $\mathrm{Gr}(\mathcal{M}_n^\d)$, which generates the $\SK$-relation for manifolds with boundary. 
    \mona{Suppose that $[M]-[N]\in R^{\partial}_n \cap \textup{Gr}(\mona{\mathcal{M}_n}).$  Since $[M]-[N]\in R^\partial_n,$ by \autoref{SKbinvariants} we have that $M\sim_{SK^\partial} N$, and thus they have the same Euler characteristic and signature. Since by \autoref{SKinvariants} the Euler characteristic and signature are the only $\SK$-invariants, it follows that $M\sim_{SK} N$. Therefore,}    \[
  R^{\partial}_n \cap \textup{Gr}(\mona{\mathcal{M}_n})=R_n,
  \]
  and injectivity of $\alpha$ follows.
% By definition,
% $$\SK_n=\textup{Gr}(M_n)/R_n \ \ \ \ \ \text{     and     }\ \ \ \ \   \SK^{\partial}_{n}=\mathrm{Gr}(M_n^\d)/R^{\partial}_n,$$
% where $R_n$ and $R^\d_n$ are the subgroups of $\mathrm{Gr}(M_n)$ and $\mathrm{Gr}(M_n^\d)$, respectively, generated by the $\SK$-relation for closed manifolds, and compact manifolds with boundary, respectively.
\raggedbottom
   \end{proof} 
\raggedbottom

\begin{figure}[ht!]
\labellist
\small\hair 2pt
%%%%%%%%%%%%    FIRST BOX
\pinlabel \small {$\simeq$} at 332 525
\pinlabel \small {$\in \SK_n^{\partial}$} at 705 525
\pinlabel \small {$+$} at 180 525
\pinlabel \small {$+$} at 527 525
%%Manifold labels first box
\pinlabel \tiny {$N$} at 93 465
\pinlabel \tiny {$DM$} at 240 465
\pinlabel \tiny {$N \cup_{\phi} \bar{M} = L $} at 437 465
\pinlabel \tiny {$\partial N \times I \cup_{\phi} M$} at 585 465
\pinlabel \tiny {$ = M$} at 585 448
%%%%%%%%%%%%    SECOND BOX
\pinlabel \small {$\simeq$} at 332 310
\pinlabel \small {$\in \SK_n^{\partial}$} at 705 310
\pinlabel \small {$-$} at 125 310
\pinlabel \small {$-$} at 488 310
%%Manifold labels second box
\pinlabel \tiny {$M$} at 60 243
\pinlabel \tiny {$N$} at 204 243
\pinlabel \tiny {$DM $} at 436 243
\pinlabel \tiny {$L$} at 581 243
%%%%%%%%%%%%    LAST BOX
\pinlabel \small {$-$} at 333 90
\pinlabel \small {$\in \im ~\alpha$} at 555 90
%%Manifold labels last box
\pinlabel \tiny {$M$} at 275 25
\pinlabel \tiny {$N$} at 401 25
\endlabellist
\centering
\includegraphics[scale=0.45]{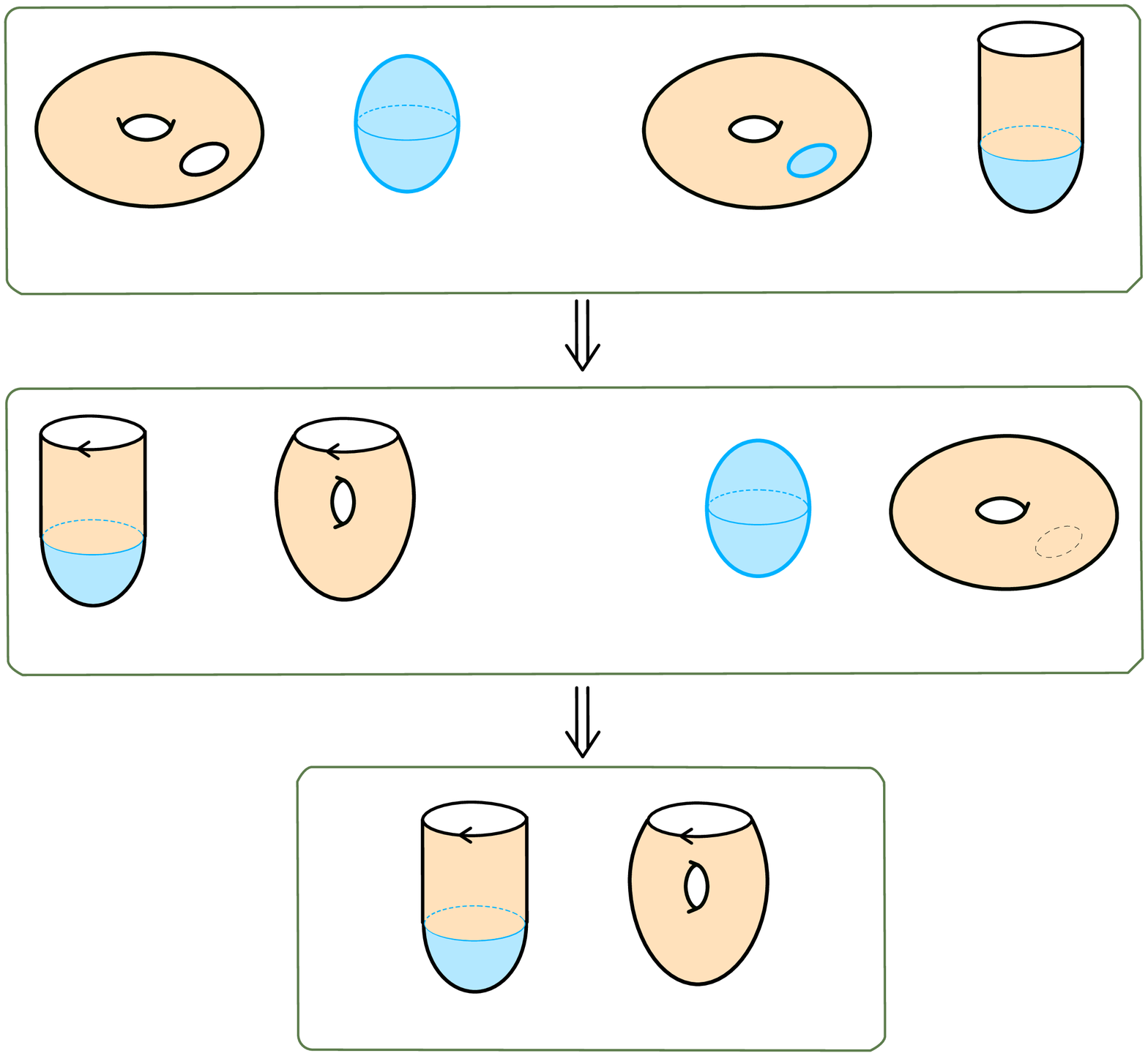}
\caption{Example of an element in $\im (\alpha \colon \SK_n \to \SK_n^{\partial})$}
\label{image-alpha}
\end{figure}
\raggedbottom

%\subsection{$\SKK$-groups for manifolds with boundary}
\mona{
\begin{sch}\label{scholium}
Classically, there is a more refined relation than that of cutting and pasting called $\SKK$ (``scheiden und kleben, kontrollierbar"=``controllable cutting and pasting") in which we keep track of the gluing diffeomorphisms. 
\laura{The SKK-equivalence relation is:}
\[ [M_1 \cup_{\phi} \bar{M'_1}] \;  - \;  [M_1 \cup_{\psi} \bar{M'_1}] \; \;  = \;  \;  [M_2 \cup_{\phi}\bar{M'_2}] \; -\; [M_2 \cup_{\psi}\bar{M'_2}]\]
 \noindent for  compact oriented manifolds $M_1, M_1'$ and $M_2, M_2'$ such that $\partial M_1=\partial M_2$ and $\partial M_1'=\partial M_2'$, and orientation preserving diffeomorphisms 
 $\phi, \psi \colon \partial M_1 \to \partial M_1'$.
The resulting $\SKK_n$-groups obtained by modding out by the $\SKK$-equivalence relation
  have been interpreted as  Reinhardt vector field bordism groups in \cite{KKNO} and have also been shown to arise as $\pi_0$ of the Madsen-Tillman spectra $MTSO(n)$  \cite{Ebert}. 
  If we define an $\SKK^\partial$ relation analogously but allow
   $M_1, M_1', M_2, M'_2$ to be manifolds with boundary, $\Sigma \subseteq \partial M_i$ and $\Sigma' \subseteq \partial M_i'$ closed submanifolds for $i=1,2$, and $\phi, \psi \colon \Sigma \to \Sigma'$ orientation preserving diffeomorphisms, then we would have
   \[ [M_1 \cup_{\phi} \bar{M'_1}] \;  - \;  [M_1 \cup_{\psi} \bar{M'_1}] \; \;  = \;  \;  [\Sigma\times I\cup_\phi \Sigma'\times I]- [\Sigma\times I\cup_\psi \Sigma'\times I], \] where the latter is zero. Thus if we defined an $SKK_n^\partial$ group by modding out by this relation, we would have $SKK_n^\partial\cong SK_n^\partial$, in contrast to the classical case where we have a surjective map $SKK_n\to SK_n$, which is not an isomorphism.\footnote{This observation is due to George Raptis.}
\end{sch}
}

    \section{$K$-theory of categories with squares}\label{squareK}

   \subsection{Overview of Campbell and Zakharevich's squares $K$-theory} 
   
   \mona{This subsection is an exposition of the definitions  and results from Campbell and Zakharevich's work in progress on $K$-theory of categories with squares.}
   %and results that we need from \cite{CZsquare}. 
       
    \begin{definition}\label{squarecat}
    	A category with squares is a category $\C$ equipped with a choice of basepoint object $O$,  two subcategories $c\C$ and $f\C$ of morphisms referred to as cofibrations (denoted \begin{tikzcd}[column sep=scriptsize]\hphantom{} \arrow[r, tail]&\hphantom{} \end{tikzcd}) and cofiber maps (denoted \begin{tikzcd}[column sep=scriptsize] \hphantom{}\arrow[r, two heads] &\hphantom{}\end{tikzcd}), and distinguished squares	
    	$$\begin{tikzcd}[column sep=tiny, row sep=tiny]
    	A\arrow[rr, tail] \arrow[dd, two heads] && B\arrow[dd,two heads]\\
    	& \square &\\
    	C\arrow[rr, tail] && D
    	\end{tikzcd}$$
    	satisfying the following conditions:
    	\begin{enumerate}[1{)}]
    		\item $\C$ has coproducts and distinguished squares are closed under coproducts.
    		\item Distinguished squares are commutative squares in $\C$ and compose horizontally and vertically.
    		\item Both $c\C$ and $f\C$ contain all isomorphisms of $\C$.
    		\item If a commutative square satisfies the property that either both horizontal maps or both vertical maps are isomorphisms, then the square is distinguished.
    	\end{enumerate}
	    \mona{A map of categories of squares is a functor that preserves distinguished basepoint objects and distinguished squares.}
    \end{definition}

   Campbell and Zakharevich developed the framework of categories with squares in order to describe a generalized construction of K-theory \mona{spaces}, inspired by the Waldhausen construction. 
%\lmar{Is this the order we want things? Or should the section on categories with squares come first, with our definition of the SK-group later (this might make more sense of why we adapt our definition to include manifolds with boundary)?}    \rnote{I agree it makes more sense to first introduce Kthy with squares and then define our category.}
%    
We review their construction of $K$-theory for a category with squares. %from \cite{CZsquare}. 
Let $[k]$ denote the category $0\to 1\to \cdots\to k$.
    
    \begin{definition}
    Let $\C$ be a category with squares. Define $\C^{(k)}$ to be the  subcategory of $\Fun([k], \C)$ whose objects are sequences of cofibration maps $$C_0\rightarrowtail C_1\rightarrowtail\cdots\rightarrowtail C_k,$$ and whose morphisms  are natural transformations in which every commutative square is distinguished.
       \end{definition}
%\lmar{I don't think this is technically a full subcategory, since we need to impose the squares condition? Maybe just a subcategory? Why did C-Z use `full subcategory' in their definition?}    
    
    Varying over $k$ by composing cofibrations and distinguished squares, we get a simplicial category, denoted $\C^\bullet$.
    The squares $K$-theory of $\C$ is defined, analogously to the definition for Waldhausen categories, as follows: 
    \begin{definition}
    Let $\C$ be a category with squares. The \emph{squares K-theory space of $\C$} is 
    $$K^\square(\C)\simeq \Omega_O |N_{\sbt}\C^\bullet|$$
    where $\Omega_O$ is the based loop space, based at the distinguished object $O\in N_0\C^{(0)}$.
   
    \end{definition}
    
     \mona{The following computation of $K_0$ for categories with squares is due to Campbell and Zakharevich; the proof will appear in upcoming work. We record the result here since we will need it later on.}
     %They also give an explicit description of the $K_0$-group for certain categories with squares, which we record here.
    
    \begin{lemma}[Campbell-Zakharevich]\label{K0}
    Let $\C$ be a category with squares with basepoint $O$ satisfying:
    \begin{enumerate}
    \item $O$ is initial or terminal in $c\C$.
    \item $O$ is initial or terminal in $f\C$.
    \item For all objects $A, B\in\C$, there exists some object $X\in\C$ and distinguished squares:
    \begin{center}
    \begin{tikzcd}[column sep=tiny, row sep=tiny]
    O\arrow[rr, tail]\arrow[dd, two heads] && A\arrow[dd, two heads]\\
    &\square &\\
    B\arrow[rr, tail] && X
    \end{tikzcd}
    \hspace{1cm}
    \begin{tikzcd}[column sep=tiny, row sep=tiny]
    O\arrow[rr, tail]\arrow[dd, two heads] && B\arrow[dd, two heads]\\
    &\square &\\
    A\arrow[rr, tail] && X
    \end{tikzcd}
    \end{center}
    \end{enumerate}
    
    Then
    $$K_0^\square(\C)\cong \mathbb{Z}\lbrace ob\C\rbrace/\sim$$
    
    where $\sim$ is the equivalence relation generated by  
    \begin{enumerate}
    	\item $[O]=0$
    	\item$[A]+[D]=[B]+[C]$ for every distinguished square
    	\begin{tikzcd}[column sep=tiny, row sep=tiny]
    		A\arrow[rr, tail] \arrow[dd, two heads] && B\arrow[dd, two heads]\\
    		&\square &\\
    		C\arrow[rr, tail] && D
    	\end{tikzcd}.

    \end{enumerate}
    \end{lemma}

  \mona{The proof that  the $K$-theory space $K^\square(\C)$ is an infinite loop space will also appear in upcoming work of Campbell, Zakharevich, and collaborators.    %By abuse of notation, we will refer to the resulting $K$-theory spectrum also as . A map of categories of squares, which is a functor that preserves distinguished basepoint objects and distinguished squares, induces a map of $K$-theory spectra.
   
%     \begin{thm}\label{spectrum}
%     Let $\C$ be a  category with squares. The space $K^\square(\C)$ is an infinite loop space.
%     \end{thm}
   }

\subsection{Category with squares from a Waldhausen category} 
\mona{The idea of Campbell's and Zakharevich's squares $K$-theory is to provide a generalization of both Waldhausen and subtractive categories. In particular, given a Waldhausen category $\C$ one can associate to it a category with squares such that the Waldhausen and squares $K$-theories agree. Actually, there are two different choices of categories with squares that one can associate to a Waldhausen category, and here we describe the one that will serve our purposes in \autoref{Eulersec}. We comment on our choices in \autoref{waldchoice} below.}
%Campbell and Zakharevich prove that squares $K$-theory is indeed a good generalization of the Waldhausen construction, in the sense that given a Waldhausen category $\C$ one can associate to it a category with squares such that the Waldhausen and squares $K$-theories agree. 
%For our purposes in \autoref{Eulersec}, we need to associate a slightly different category with squares to a Waldhausen category than that defined in \cite{CZsquare}. We will show that the Waldhausen $K$-theory and squares $K$-theory are also compatible in this case; the proof is completely analogous to the proof given in \cite{CZsquare}, but we include the version for our particular case here for completeness. We comment on our choices in \autoref{waldchoice} below.

\begin{defn}\label{squareC}
Let $\sC$ be a Waldhausen category. Define an associated category with squares $\sC^\square$ in the following way. The horizontal maps are the cofibrations $\rightarrowtail$ in $\sC$, and the vertical maps are all maps. The distinguished squares are the squares 
\[ \begin{tikzcd}[column sep=tiny, row sep=tiny]
A\arrow[rr, tail] \arrow[dd] && B\arrow[dd]\\
&\square &\\
C\arrow[rr, tail] && D
\end{tikzcd}
\]
with the property that the unique map $C\cup_A B\rightarrow D$ is a weak equivalence. The distinguished basepoint object is the zero object.
\end{defn}

\begin{prop}
The category $\sC^\square$ satisfies the axioms of a category with squares from \autoref{squarecat}. 
\end{prop}

\begin{proof}
We check the four axioms. For (1), $\sC$ has coproducts because it is a Waldhausen category. Suppose that 
$$C\cup_A B\xrightarrow{\simeq} D \text{   and   } C'\cup_{A'} B'\xrightarrow{\simeq} D'.$$ Note that since pushouts and coproducts commute with each other, and since $$C\cup_A B \sqcup C'\cup_{A'} B' \xrightarrow{\simeq} D\sqcup  D'$$ by the gluing axiom (\cite[p. 326]{waldhausen}), distinguished squares are closed under coproducts.

To check axiom (2), suppose we compose two distinguished squares horizontally
\[ \begin{tikzcd}[column sep=tiny, row sep=tiny]
A\arrow[rr, tail] \arrow[dd] && B\arrow[rr, tail]\arrow[dd] && E\arrow[dd]\\
& \square && \square \\
C\arrow[rr, tail] && D\arrow[rr, tail] && F
\end{tikzcd} \]
We have a chain of weak equivalences $$C\cup_A E \cong (C\cup_A B)\cup_B E \xrightarrow{\simeq} D\cup_B E \xrightarrow{\simeq} F,$$ where the first weak equivalence is by the gluing axiom.  

Now suppose we compose two distinguished squares vertically
\[ \begin{tikzcd}[column sep=tiny, row sep=tiny]
A\arrow[rr, tail] \arrow[dd] && B\arrow[dd] \\
& \square & \\
C\arrow[rr, tail]\arrow[dd] && D\arrow[dd]\\
& \square &\\
E\arrow[rr, tail] && F
\end{tikzcd} \]

Similarly, we have $$E\cup_A B\cong E\cup_C(C\cup_A B)\xrightarrow{\simeq} E\cup_C D \xrightarrow{\simeq} F, $$ where again the first weak equivalence is by the gluing axiom.

Axiom (3) is immediate since the isomorphisms are contained in the cofibrations in a Waldhausen category, and we don't have any restrictions on the vertical maps.

To check axiom (4), suppose first that the two vertical morphisms in a commuting square in $\sC$
\[ \begin{tikzcd}[column sep=tiny, row sep=tiny]
A\arrow[rr, tail] \arrow[dd] && B\arrow[dd]\\
&\square &\\
C\arrow[rr, tail] && D
\end{tikzcd}
\]
are isomorphisms. Then $C\cup_A B\cong B\cong D$ and the square is a pushout square. Similarly, if the horizontal maps are isomorphisms, $C\cup_A B\cong C\cong D$, and again the square is a pushout.
\end{proof}

\medskip

\begin{prop}\label{sqwald=wald}
The Waldhausen $K$-theory $K^\mathrm{Wald}(\sC)$  agrees with the $K$-theory $K^\square(\sC^\square)$ of the associated category with squares from \autoref{squareC}.
\end{prop}

\begin{proof}
By definition, $K^\square(\sC^\square)$ is the realization of the bisimplicial set with $(p, q)$-simplices given by 

\[
\begin{tikzcd}[column sep=tiny, row sep=tiny]
A_{00}\arrow[rr, tail]\arrow[dd] && A_{01}\arrow[rr, tail]\arrow[dd] &&\cdots\arrow[rr, tail] && A_{0p}\arrow[dd]\\
& \square && \square && \square\\
A_{10}\arrow[rr, tail]\arrow[dd] && A_{11}\arrow[rr, tail]\arrow[dd] &&\cdots\arrow[rr, tail] && A_{1p}\arrow[dd]\\
& \square &  & \square  & & \square \\
\vdots\arrow[dd] && \vdots\arrow[dd] && \ddots &&\vdots\arrow[dd]\\
& \square & & \square  & & \square \\
A_{q0}\arrow[rr, tail] && A_{q1}\arrow[rr, tail] &&\cdots\arrow[rr, tail] && A_{qp}
\end{tikzcd}
\]

\noindent in which each square is distinguished. Thus it is the nerve of the category whose objects are sequences of cofibrations $$A_0 \rightarrowtail A_1 \rightarrowtail \dots  \rightarrowtail  A_n$$ and morphisms maps of such diagrams that satisfy the condition that for every $i \leq j$ the induced map
$$A_i'\cup_{A_i} A_j \to A_j'$$ is a weak equivalence. 
%\mmar{is this the same with requiring this condition on each square, why does Thomason require it for all $i\leq q$? or does this follow for us from axiom 2?} 
Thus the above is precisely the bisimplicial set obtained by applying the  nerve to Thomason's simplicial category $wT\sbt\ \!\sC$ defined in \cite[page 334]{waldhausen}.

By Thomason-Waldhausen, there is a zig-zag of equivalences via some intermediate construction
\[ \xymatrix{ wT\sbt\ \! \sC && wT^+\sbt\ \!\sC \ar[ll]_-\simeq \ar[rr]^-\simeq && wS\sbt \ \!\sC.}\]
Therefore, via a zig-zag, we have an equivalence of $K$-theory spaces
 \begin{equation*}K^\square(\sC^\square)\simeq K^\mathrm{Wald} (\sC). \qedhere\end{equation*} \raggedbottom
\end{proof}
\raggedbottom
 \medskip

\begin{rem}\mona{
Campbell and Zakharevich noticed that for a Waldhausen category $\sC$, one can also associate to it a category with squares where the vertical maps are cofiber maps, and they prove directly that the square $K$-theory of this category is equivalent to $K^\mathrm{Wald}(\sC)$. Therefore, it is also equivalent to the square $K$-theory of the category from   \autoref{squareC}.}
%The category with squares associated to the Waldhausen category $\sC$ in \autoref{squareC} is different from the category with squares associated to $\sC$ in \cite[Example 1.2.]{CZsquare}. However, they have equivalent $K$-theories since they are both equivalent to the usual Waldhausen $K$-theory $K^\mathrm{Wald}(\sC)$. For the category with squares from  \cite[Example 1.2.]{CZsquare}, this is proved directly  in  \cite[Lemma 1.5.]{CZsquare}. 
\end{rem}

\section{$K$-theory of manifolds  with boundary}\label{mnfldsec}
In this section we use the framework described in \autoref{squareK} to define a $K$-theory \mona{space} for the category of $n$-dimensional compact smooth manifolds with boundary, which recovers as $\pi_0$ the scissors congruence group $\SK_n^\partial$. 

\subsection{The category with squares for manifolds with boundary} We start by defining a category with squares structure on the category  $\Mnfldbd_n$ of  smooth compact $n$-dimensional manifolds with boundary and smooth maps.

    \begin{definition}\label{Mnfsq}
Let $\Mnfldbd_n$ be the category of smooth compact $n$-dimensional manifolds with boundary and smooth maps. We define the subcategories  $c\Mnfldbd_n$ of horizontal maps (denoted $\rightarrowtail$) and and $ f\Mnfldbd_n$  of vertical maps (denoted $\hookrightarrow$) to both be given by the morphisms in $\Mnfldbd_n$ which are smooth embeddings of manifolds with boundary \carmen{$f\colon N \rightarrow M$} such that \carmen{$\partial N$} is mapped to a submanifold with trivial normal bundle, and such that each connected component of the boundary \carmen{$\partial N$} is either mapped entirely onto a boundary component or entirely into the interior of \carmen{$M$}.
We define distinguished squares to be those commutative squares in $\Mnfldbd_n$     	
$$\begin{tikzcd}
    	N	\arrow[r, tail]\arrow[d, hook] &M\arrow[d, hook]\\
    	M'\arrow[r, tail] & M \cup_{N} M'.
    	\end{tikzcd}$$ that are pushout squares, i.e. such that $ M\cup_{N} M'$ is a smooth manifold. The chosen basepoint object is the empty manifold.
    \end{definition}

   \begin{exmp}
\autoref{SK-pic} gives pictorial examples of distinguished squares in $\Mnfldbd_n$ .
\end{exmp}

\begin{figure}[ht!]
\labellist
\small\hair 2pt
%%%%%%%%%%%%    FIRST square
\pinlabel \small {$\square$} at 140 150
%%%%%%%%%%%%    SECOND square
\pinlabel \small {$\square$} at 610 150
\endlabellist
\centering
\includegraphics[scale=0.5]{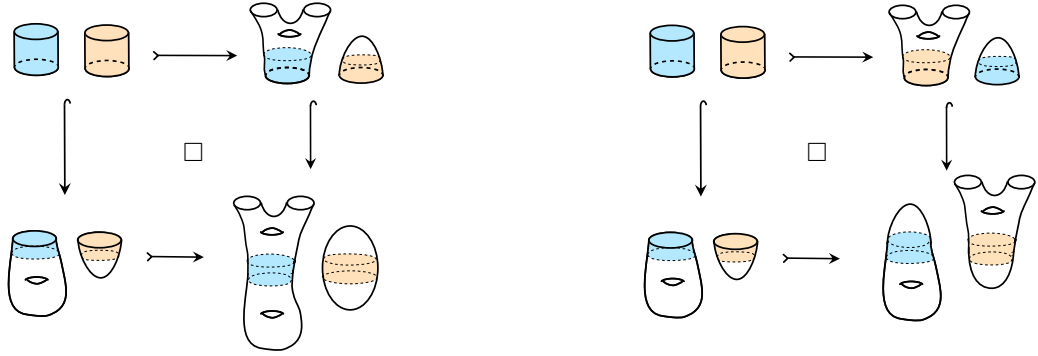}
\caption{Two examples of distinguished squares} 
\label{SK-pic}
\end{figure}

  %  \mmar{see discussion about changes to this pic in previous draft}

%\rnote{Perhaps it would be nice to highlight the images of the two cylinders rather than just the boundary so that it's clear we're gluing along the whole cylinders.}
%\lnote{I like Renee's suggestion. Would it also be clearer to color the cylinders vs. the arrows, to show how the twisting works? (I guess this would require two separate pictures.) Or should we label one of the top horizonal arrows as a `twist' or `order reversal'? It's clear to us what this diagram is representing, but I'm not so sure it would be clear for someone else first learning these things.}
%\rnote{I like the suggestion of colouring the cylinders! Also we decided to split up the picture into two diagrams.}   

    \newpage
    \begin{lemma}
    The category	$\Mnfldbd_n$ with the structure from \autoref{Mnfsq} satisfies the axioms of a  category with squares from  \autoref{squarecat}.
    \end{lemma}
    
    \begin{proof}
    	The coproduct in  $\Mnfldbd_n$ is given by disjoint union of manifolds, and the collection of distinguished squares is closed under disjoint union. Pushout squares are commutative and compose horizontically and vertically. Consider the diagram	
    	$$\begin{tikzcd}
    	A\arrow[r, tail, "i"] \arrow[d, hook, "j"]  &  B \arrow[d, hook, "j'"] \arrow[ddr, bend left, "f"] & \\
    	C 	\arrow[r, tail, "i'"]  \arrow[drr, bend right, "f'"] & D \arrow[dr, dotted, "l"] & \\
    	& & R.	\end{tikzcd}$$
    	If $j'$ is an isomorphism then we can define the map $l$ uniquely as $f j'^{-1}$; similarly if $i'$ is an isomorphism. Therefore in both cases this is a pushout diagram. Hence $\Mnfldbd_n$ satisfies the definition of a category with squares. 	
    \end{proof}

\subsection{The computation of $K_0(\Mnfldbd_n)$} Using \autoref{K0} for the category with squares $\Mnfldbd_n$ defined above, we show that the $K^\square_0$-group agrees with the $\SK_n^\d$-group. 
    
    \begin{thm}
    	For  the manifold category with squares $\Mnfldbd_n$ from \autoref{Mnfsq}, 
    	$$K_0^\square(\Mnfldbd_n)\cong \SK_n^\d.$$
    \end{thm}
    
    \begin{proof}

    		The empty set is initial in both $c\Mnfldbd_n$ and $f\Mnfldbd_n$. Moreover, for all objects $M$ and $N$ in $\Mnfldbd_n$, there exist pushout squares
    	\begin{center}
    		\begin{tikzcd}[column sep=tiny, row sep=tiny]
    			\emptyset	\arrow[rr, tail]\arrow[dd, hook] &&N\arrow[dd, hook]\\
    			&~~\square &\\
    			M\arrow[rr, tail] && M\sqcup N
    		\end{tikzcd}
    		\hspace{1cm}
    		\begin{tikzcd}[column sep=tiny, row sep=tiny]
    			\emptyset	\arrow[rr, tail]\arrow[dd, hook] &&M\arrow[dd, hook]\\
    			&~~~\square &\\
    			N\arrow[rr, tail] && M\sqcup N.
    		\end{tikzcd}
    	\end{center}
    	Therefore $\Mnfldbd_n$ satisfies the conditions of \autoref{K0}, which gives a description of the relations of the left hand side.

    	First, assume that the relations from $K_0^\square$ hold. To show that these imply the relations in $\SK_n^\d$, we first need to check that the generating objects are compatible (note that  $\SK_n^\d$ is generated by diffeomorphism classes of manifolds, whereas $K_0$ a priori is generated by manifolds). Consider a diffeomorphism $M\xrightarrow{\phi} M'$. Then
    	\begin{center}
    		\begin{tikzcd}[column sep=tiny, row sep=tiny]
    			\emptyset\arrow[rr, tail]\arrow[dd, hook] && M\arrow[dd, hook, "\phi"]\\
    			&~\square &\\
    			\emptyset\arrow[rr, tail] && M'
    		\end{tikzcd}
    	\end{center}
    	is a distinguished square; and so the relations in $K_0^\square$ give that:
    	\begin{align*}
    	[M]+[\emptyset] &= [M']+[\emptyset]\\
    	[M] &= [M']
    	\end{align*}
    	
    	Next, consider the square
    	\begin{center}
    		\begin{tikzcd}[column sep=tiny, row sep=tiny]
    			\emptyset\arrow[rr, tail]\arrow[dd, hook] && M\arrow[dd, hook]\\
    			&~~\square &\\
    			M'\arrow[rr, tail] && M\sqcup M'.
    		\end{tikzcd}
    	\end{center}
    	This is a distinguished square, which means that
    	\begin{align*}
    	[M]+[M'] &= [M\sqcup M']+[\emptyset]\\
    	&=[M\sqcup M'].
    	\end{align*}
    	
    	For the other relation in $\SK_n^\d$, consider
    	compact oriented manifolds $M, M'$, closed submanifolds $\Sigma \subseteq \partial M$ and $\Sigma' \subseteq \partial M'$, and orientation preserving diffeomorphisms $\phi, \psi \colon \Sigma \to \Sigma'$. We want to show that
    	\[ [M \cup_{\phi} M'] =[M \cup_{\psi} M'].\]
    	
    	Consider $(\Sigma\times \epsilon)$ where $\epsilon = [0,\varepsilon]$ for some small $\varepsilon > 0$.
    	We can extend the maps $\phi, \psi$ by the identity to maps $\tilde{\phi}, \tilde{\psi}$ from $(\Sigma\times \epsilon)$ to $(\Sigma'\times \epsilon)$, which we consider inside $M$ and $M'$ respectively as collars of the boundary components. This is possible as the boundary has trivial normal bundle. We have that $M \cup_{\phi} M'$ is diffeomorphic to $M \cup_{\tilde{\phi}} M'$. 
    	Using the maps $\phi, \psi$, consider the squares
    	\begin{center}
    		\begin{tikzcd}[column sep=tiny, row sep=tiny]
    			(\Sigma\times\epsilon)\arrow[rr, tail]\arrow[dd, hook, "\tilde{\phi}"'] && M\arrow[dd, hook]\\
    			&\square &\\
    			M'\arrow[rr, tail] && M\cup_\phi M'
    		\end{tikzcd}
    		\hspace{1cm}
    		\begin{tikzcd}[column sep=tiny, row sep=tiny]
    			(\Sigma \times\epsilon)\arrow[rr, tail]\arrow[dd, hook, "\tilde{\psi}"'] && M\arrow[dd, hook]\\
    			&\square &\\
    			M'\arrow[rr, tail] && M\cup_\psi M'.
    		\end{tikzcd}
    	\end{center}
    	
    	%\textcolor{red}{When we use $\phi, \psi$, are we still dealing with inclusions (vertical arrows in our set-up)? How do we know these are distinguished squares? Assuming that they are distinguished squares, then...} \textcolor{purple}{Note that these are pushout squares and are therefore distinguished in our category $\C$.}
    	The relation given by distinguished squares implies:
    	\begin{align*}
    	[M\cup_\phi M'] +[(\Sigma \times \epsilon)] &= [M]+[M']\\
    	&= [M\cup_\psi M'] +[(\Sigma \times \epsilon)]
    	\end{align*}
    	Thus, $[M\cup_\phi M']=[M\cup_\psi M']$.
    	
    	In the other direction, assume that the relations for $\SK_n^\d$ hold. Consider relation (1) in \autoref{def of SK_n with boundary} applied to the following:
    	\begin{align*}
    	[\emptyset\sqcup \emptyset] &= [\emptyset]+[\emptyset]\\
    	[\emptyset]&= [\emptyset]
    	\end{align*}
   
    	Thus, for $\emptyset$, the initial object in our category with squares, we have $[\emptyset]=0$.
    	
    	Finally, for relation (2) of \autoref{def of SK_n with boundary}, suppose the following is a distinguished square:
    	\begin{center}
    		\begin{tikzcd}[column sep=tiny, row sep=tiny]
    			A\arrow[rr, tail] \arrow[dd, hook] && B\arrow[dd, hook]\\
    			&\square &\\
    			C\arrow[rr, tail] && D
    		\end{tikzcd}
    	\end{center}
Define $N:= A\cap cl(B-A)\subseteq \d A$, where $cl(B-A)$ is the closure of the complement of $A$ in $B$, i.e. $N$ is the part of the boundary of $A$ that is mapped to the interior of $B$.  	
    	We define $$M:= cl(B-A)\sqcup (N\times \epsilon),$$ $$M':= A\sqcup C.$$ Let $\id: N\sqcup N\to N\sqcup N$ be the identity map; let $\tau: N\sqcup N\to N\sqcup N$ be the twist map.
    	%\textcolor{red}{[Is there a more clear/formal way to phrase this?]}
    	Note that $M\cup_\id M'\cong B\sqcup C$ and $M\cup_\tau M'\cong A\sqcup D$. Then the fact that $[M\cup_\id M']=[M\cup_\tau M']$ gives the relations
    	\begin{align*}
    	[B\sqcup C]&=[A\sqcup D],\\
    	[B]+[C] &=[A]+[D]. \qedhere
    	\end{align*} 
    \end{proof}

    \section{The derived Euler characteristic for manifolds with boundary}\label{Eulersec}
    
 The Euler characteristic map $\chi\colon \mathcal{M}_n^\d \to \Z$ from the monoid of diffeomorphism  classes of smooth compact manifolds is an $\SK$-invariant, since $\chi(M\cup_\Sigma N)=\chi(M)+\chi(N)-\chi(\Sigma)$; thus it factors through $\SK_n^\d$. We show that the Euler characteristic map $\chi\colon \SK_n^\d\to \Z$ lifts to a map of \mona{$K$-theory spaces}. The strategy will be to construct a map of categories with squares from the category of smooth compact oriented manifolds with boundary to the category with squares from \autoref{squareC} associated to the Waldhausen category of \julia{homologically bounded} $\Z$-chain complexes. The main theorem we prove in this section is the following.

 \begin{thm}\label{Eulerthm}
 There is a map of $K$-theory \mona{spaces} $$K^\square(\Mnfldbd)\to K(\Z),$$ which on $\pi_0$ agrees with the Euler characteristic for smooth compact manifolds with boundary.
 \end{thm}
 
We first prove the propositions we need in the next section and give the proof of the theorem at the end of the final section.

\subsection{The lift of the singular chain functor}
Let $\ChPerf_\mathbb{Z}$ be the Waldhausen category of \julia{homologically bounded chain complexes}, i.e., those complexes that are quasi-isomorphic to bounded finitely generated $\Z$-complexes, 
%Let $\ChQ$ be the Waldhausen category of homologically bounded $\Q$-chain complexes whose homology is finitely generated, 
with cofibrations given by levelwise injective maps and weak equivalences given by quasi-isomorphisms. Consider the associated category with squares $(\ChPerf_\mathbb{Z})^\square$ as defined in \autoref{squareC}.

Consider the singular chain functor 
$$S \colon \Mnfldbd_n\to \ChPerf_\mathbb{Z}$$ which sends a compact manifold with boundary to its singular chain complex. The homology of this complex is finitely generated in each degree and bounded since our manifolds are compact.

\begin{prop}\label{singfunctor}
The map $S$ is a map of categories with squares $$S \colon \Mnfldbd_n\to (\ChPerf_\mathbb{Z})^\square$$
\end{prop}

\begin{proof}
Suppose we have a distinguished square 
\[ \begin{tikzcd}[column sep=tiny, row sep=tiny]
A\arrow[rr, tail] \arrow[dd, hook] && B\arrow[dd,hook]\\
&\square &\\
C\arrow[rr, tail] && D
\end{tikzcd}
\]

\noindent in $\Mnfldbd$, and we apply $S$ to it. In the resulting square in $\ChPerf_\mathbb{Z}$, the horizontal maps are levelwise injective, as required. So in order to show that it is a distinguished square, it remains to show that the map $$S(B)\cup_{S(A)}S(C) \to S(D)$$ 
% vs. $$S(A)\cup_{S(A)}S(B) \to S(D)$$ 
is a quasisomorphism.

Note that by our construction of distinguished squares in $\Mnfldbd$ the union of the interiors of $B$ and $C$ covers $D.$ Let $S_n(B + C)$ be the subgroup of $S_n(D)$ consisting of $n$-chains that are sums of $n$-chains in $B$ and $n$-chains in $C$. By the standard Mayer-Vietoris argument, the following sequence is exact

\begin{center}
	$\begin{CD}
	0 @>>>  S_n(A) @>>{x \mapsto (x, -x)}> S_n(B) \oplus S_n(C) @>>{(y,z) \mapsto y+z}> S_n (B+C) @>>> 0
	\end{CD}.$
\end{center} 
Hence the chain complex $S_*(B+C)$ is a pushout $S_*(B)\cup_{S_*(A)} S_*(C)$. On the other hand by \cite[Proposition 2.21]{hatcher}, the inclusions $S_n (B+C) \to S_n(D)$ induce isomorphisms on homology groups, which finishes the proof.
\end{proof}

\begin{rem}\label{waldchoice}
The reason for the choices in our \autoref{squareC} of a category with squares associated to a Waldhausen category is precisely to make the above proposition work. 
%The difference between the category with squares in \autoref{squareC} and that in \cite[Example 1.2]{CZsquare} is that we allow the Waldhausen category $\sC$ to have weak equivalences and not only isomorphisms, so we can apply it to the category of chain complexes, and
\mona{Note that we allow \emph{all} maps as vertical maps as opposed to only the cofiber maps, which is the other way one could associate a category with squares to a Waldhausen category.} This more relaxed definition of the distinguished squares is crucial in allowing us to show that distinguished squares in the category of manifolds map to distinguished squares in the category of chain complexes. 
\end{rem}

\subsection{Recovering the Euler characteristic on $\pi_0$}
Lastly, we claim that $K(\ChPerf_\mathbb{Z})\simeq K(\Z)$ via an isomorphism under which $S(M)$ corresponds to $\chi(M)$ on $\pi_0$ for a smooth compact oriented manifold $M$.

 Denote by $\ChB_\mathbb{Z}$ the category of bounded complexes of finitely generated $\mathbb{Z}$-modules; the homologically complexes are those that are quasi-isomorphic to complexes in $\ChB_\mathbb{Z}$. Note that $\ChB_\mathbb{Z} \subseteq \ChPerf_\mathbb{Z}$ and moreover by the discussion in \cite[V. 2.7.2]{kbook} (or alternatively directly by the Waldhausen approximation theorem) this inclusion\julia{, which we denote by $j$}, induces an isomorphism \julia{$j_*$} on $K$-groups. A similar argument for cohomology appears in \cite[Lemma 2.8]{CWZ}.
 
\julia{Consider the following chain of maps  of $K$-theory spaces
	
 \begin{center}
 	$\begin{CD}
 K(\Z)=K(\ModPROJ_{\Z}) @>{\simeq}>{i_*}> G(\Z)=K(\ModFG_{\Z}) @>{\simeq}>{t_*}> K(\ChB_{\Z}) @>{\simeq}>{j_*}> K(\ChPerf_{\Z}), 
 	\end{CD}$
\end{center} 
 where $\ModPROJ_{\Z}$ is the category of finitely generated projective $\Z$-modules, $\ModFG_{\Z}$ is the category of finitely generated $\Z$-modules. Let $i_*$ be the map induced by the inclusion of categories $i \colon \ModPROJ_{\Z} \to \ModFG_{\Z}.$ Let $t \colon \ModFG_{\Z} \to \ChB_{\Z}$ be the inclusion of categories that sends a module $M$ to the chain complex with $M$ in degree $0$ and zeroes in all other degrees.
For a regular noetherian ground ring, in particular for $\Z$, the map $i_*$ is an isomorphism on K-groups by the resolution theorem \cite[Theorem V.3.3.]{kbook}. By the Gillet-Waldhausen theorem \cite[Theorem V.2.2.]{kbook} $t_*$ is a homotopy equivalence and hence induces an isomorphism on K-groups. The map $j_*$ is also a homotopy equivalence by the Waldhausen approximation theorem \cite[V.2.7.4.]{kbook}.
 
 The next proposition describes an inverse of $j_*$ on the zeroth K-groups.
 
}

\begin{prop}\label{k0prop}
The map $q \colon K_0 (\ChPerf_\mathbb{Z}) \to K_0 (\ChB_\mathbb{Z})$ sending a chain complex $C_*$ to the class of the corresponding quasi-isomorphic chain complex $H(C_*) \in \ChB_\mathbb{Z}$ is well-defined and is an isomorphism.
\end{prop}
\begin{proof}
%\textit{Well-defined.} If $C_*\sim B_*$ and $C_*\sim D_*$ with $B_*, D_* \in \ChB (\mathbb{Q},$ then clearly $B_*$ and $D_*$ are also quasi-isomorphic, and hence represent the same class in $K_0 \ChB (\mathbb{Q}).$\\
The map is well-defined since quasi-isomorphic chain complexes have isomorphic homology, and it
%\textit{Isomorphism.} It is obvious that $q \colon K_0 \ChPerf (\mathbb{Q}) \to K_0 \ChB (\mathbb{Q})$ 
is surjective because of the inclusion $\ChB_\mathbb{Z} \subseteq \ChPerf_\mathbb{Z}.$ On the other hand if $Y=q(X)$ vanishes in $K_0 (\ChB_\mathbb{Z})$ then we may identify $X$ with $Y$ in $K_0 (\ChPerf_\mathbb{Z})$ and it will also vanish there, because the set of defining relations (which we quotient out in the presentation for $K_0$) of $K_0( \ChPerf_\mathbb{Z})$ contains the defining relations of $K_0 (\ChB_\mathbb{Z}).$
\end{proof}

\medskip 

Now, recall that the map $\phi \colon K_0( \ChB_\Z) \to K_0(\Z)$ given by $[C_*] \mapsto \chi(C_*)=\sum_i (-1)^i [C_i] $ is an isomorphism \cite[Proposition II.6.6.]{kbook} \julia{and coincides with the composition $ i_{*}^{-1} \circ t_*^{-1}$ on the zeroth K-groups \cite[Theorem II.9.2.2.]{kbook}}. %By proposition II.6.6 in Weibel 
 By an easy exercise using the additivity property, the Euler characteristic of a bounded complex only depends on its homology and  $\chi(C_*)=\sum_i (-1)^i [H_i(C_*)].$ 
Thus the composition $\phi \circ q \colon K_0 (\ChPerf_\mathbb{Z} )\to K_0(\Z)$ 
is also an isomorphism and maps $[C_*]$ to $\chi(C_*)=\sum_i (-1)^i [H_i(C_*)]$. \julia{It coincides with the composition $ i_{*}^{-1} \circ t_*^{-1} \circ j_*^{-1} \colon  K_0 (\ChPerf_\mathbb{Z} )\to K_0(\Z).$}

 \begin{proof}[Proof of \autoref{Eulerthm}.]
 From \autoref{singfunctor}, the singular chain functor $S \colon \Mnfldbd_n\to (\ChPerf_\mathbb{Z})^\square$ is a map of categories with squares when the right hand side is given the structure of a category with squares from \autoref{squareC}. This then induces a map on \mona{$K$-theory spaces}
  $$K^\square(\Mnfldbd)\to K^\square\big((\ChPerf_\mathbb{Z})^\square\big).$$
  By \autoref{sqwald=wald} the target is  $K(\Z)$. By \autoref{k0prop} and the  discussion following it, the map on $\pi_0$ agrees with the Euler characteristic. 
 \end{proof}
%
%\mona{
%\begin{rem}
%We note that by construction, the map $K^\square(\Mnfldbd)\to K(\Z)$ factors through $A(\ast)$, the Waldhausen $A$-theory of a point. 
%\end{rem}
%}
%\medskip

%\begin{thebibliography}{99}
%\bibitem[K]{K} Kreck

%\end{thebibliography}

 \bibliographystyle{amsalpha}
  \bibliography{biblio}

\begingroup%
\setlength{\parskip}{\storeparskip}% Restore \parskip within this scope

\end{document}